\documentclass{amsart}

\usepackage{amssymb}
\usepackage{mathabx}
\usepackage{tikz}
\usepackage{rotating}
\usepackage{ytableau}
\usepackage{mathptmx}     
\usepackage{helvet}   
\usepackage{courier}
\usepackage{amsfonts} 

\usepackage{mathtools}

\newtheorem{theorem}{Theorem}[section]

\theoremstyle{definition}
\newtheorem{definition}[theorem]{Definition}
\newtheorem{example}[theorem]{Example}

\theoremstyle{proposition}
\newtheorem{proposition}[theorem]{Proposition}

\theoremstyle{corollary}
\newtheorem{corollary}[theorem]{Corollary}

\numberwithin{equation}{section}

\begin{document}

\title[Lattice path enumeration]%
{Lattice path enumeration for \\ semi-magic squares of size three} 


\author{Robert W. Donley, Jr.}
\address{Department of Mathematics and Computer Science,  Queensborough Community College (CUNY), Bayside, NY 11364, USA}
\curraddr{}
\email{RDonley@qcc.cuny.edu}


\subjclass[2010]{Primary 05E10 05E18; Secondary 15B51 81R05} 
\keywords{3-$j$ symbol, Clebsch-Gordan coefficient, Franel numbers, lattice path, Regge symmetry, semi-magic square, Vandermonde convolution, Wigner symbol, wreath product}

\begin{abstract}  
We give methods for enumerating lattice paths in the graded poset of semi-magic squares of size three.   Two applications of these formulas are given:  an advanced example of Vandermonde convolution for finite graded posets, and a direct method for deriving Regge symmetry formulas for un-normalized Clebsch-Gordan coefficients.
\end{abstract}

\maketitle

\section{Introduction}

Semi-magic squares of size three appear as a fundamental feature in the quantum theory of angular momentum.  When the indices of the $3$-$j$ symbol are reconfigured in square form as the Regge symbol \cite{Re1}, the usual symmetries of the determinant account for all possible symmetries of the indices of the coupled system. While somewhat unexpected at first glance, the semi-magic property arises naturally in the indexing of certain sets of homogenous polynomials of fixed degree.

Extending the ideas of \cite{Do3}, we pursue further combinatorial connections between semi-magic squares and 3-$j$ symbols as Clebsch-Gordan coefficients; other than general background and notational concerns, this work is self-contained until Section 7, where it will be convenient to be familiar with the indexing of \cite{Do1}. 

In this work, a description of semi-magic squares using permutation matrices allows for diagrammatic indexing.   In summary, Section 2 gives an overview of semi-magic squares, noting both algebraic features and  properties from graded posets.  Sections 3 and 4 define and develop the permutation indexing, based on certain actions of wreath products on rectangles. Section 5 presents formulas for lattice path counting, and Sections 6 and 7 apply these formulas.  Section 6 gives an advanced example of Vandermonde  convolution for graded posets developed in \cite{Do4}, and Section 7 gives a direct, efficient derivation of the Regge symmetries, as formulated in \cite{Do1}.  A defining feature of the traditional 3-$j$ symbol is that Regge symmetries act simply by sign changes, a property lost in the un-normalized formulation.
 
The approach in Section 7 is somewhat makeshift. Computing the path formula in multinomial coefficient form, we see that it incorporates a hypergeometric factor of type ${}_3F_2;$  it is well known that Clebsch-Gordan coefficients also do so.  Restoring the variable establishes a direct link.  Combinatorial reciprocity suggests the missing philosophical link, although further interpretation is required.  

\section{Semi-magic squares and the poset $M(r, s)$}

We begin with general features of semi-magic squares.  

\begin{definition} Fix an integer $r\ge 1.$
A square matrix $M$ of size $r$ with non-negative integer entries is called a {\bf semi-magic square} if the sum along any row or column equals the same number $\rho(M)$.  Then $\rho(M)$ is called the {\bf line sum} of $M.$   
\end{definition}

If $k\ge 0$ and $M, N$ are semi-magic squares, then $kM+N$ is a semi-magic square and 
$$\rho(kM+N) = k\rho(M)+\rho(N).$$
We also note that $MN$ is also a semi-magic square with $\rho(MN)=\rho(M)\rho(N),$ although this fact is not needed in this work.

\begin{definition}
Let $M(r)$ be the monoid of all semi-magic squares of size $r$.  For each $s\ge 0,$ define $M(r, s)$ to be the subset of semi-magic squares where each entry is less than or equal to $s.$
\end{definition} 

Every permutation matrix  of size $r$ is a semi-magic square with $\rho(M)=1,$ and it follows from the Birkhoff-von Neumann theorem  (\cite{Bkf}, \cite{VN}) that every element of $M(r)$ with line sum $L$ admits a representation as a sum of $L$ permutation matrices of size $r$.  This representation need not be unique.  In particular, for any $r$, we denote by $J$ be the semi-magic square with all entries equal to 1; the non-uniqueness  of representing $J$ in this way is crucial to what follows.

\begin{definition}   Define a partial ordering on $M(r)$ by $M \le N$ if and only if  $m_{ij} \le n_{ij}$ for all $1\le i, j \le r$.  The line sum $\rho(M)$ gives a rank function  on $M(r)$ with covering relation given by $M \precdot N$ if and only if $N=M+P$ for some permutation matrix $P$ of size $r$.  The unique minimum element of $M(r)$ is $0J$.
\end{definition}

With the induced ordering and rank function,  $M(r, s)$ is a finite graded poset with $\hat{0}$ and $\hat{1}$, as defined in \cite{StE}; we note that
\begin{enumerate}
\item the unique minimum element $\hat{0} = 0J$,
\item the unique maximum element $\hat{1} = sJ,$
\item a dual ordering is induced by the involution $M^* = sJ-M,$ and 
\item $\rho(M^*) = rs-\rho(M).$
\end{enumerate}

For purposes of lattice path counting, we may consider $M(r)$  as a subset of $\mathbb{Z}^{r^2}\subset \mathbb{R}^{r^2}$ in the usual manner for matrices. Here lattice paths start at the origin with allowable steps given by the $r!$ permutation matrices. In the language of graded posets, a lattice path to $M$ corresponds to a maximal chain in the order ideal $I(M)$, where  $I(M)$ is the  induced subposet consisting of elements less than or equal to $M$ in $M(r).$ 

\section{An alternative description of  $M(3)$ and $M(3, s)$}

We now consider the case of $r=3.$  It will be convenient to quantify the non-uniqueness of permutation matrix representations of these semi-magic squares using rectangles associated to a wreath product action.

We fix notation for the permutation matrices of size 3:
$$P_1=P_e = \begin{bmatrix} 1 & 0 & 0 \\ 0 & 1 & 0 \\ 0 & 0 & 1 \end{bmatrix},\qquad P_2=P_{(123)}= \begin{bmatrix}   0 & 0 & 1 \\ 1 & 0 & 0 \\ 0 & 1 & 0\end{bmatrix},\qquad P_3=P_{(132)}= \begin{bmatrix}   0 & 1 & 0 \\ 0 & 0 & 1 \\ 1 & 0 & 0\end{bmatrix},$$

$$P_4=P_{(13)} = \begin{bmatrix} 0 & 0 & 1 \\ 0 & 1 & 0 \\ 1 & 0 & 0 \end{bmatrix},\qquad P_5=P_{(12)}= \begin{bmatrix}   0 & 1 & 0 \\ 1 & 0 & 0 \\ 0 & 0 & 1\end{bmatrix},\qquad P_6=P_{(23)}= \begin{bmatrix}   1 & 0 & 0 \\ 0 & 0 & 1 \\ 0 & 1 & 0\end{bmatrix}.$$

As noted, every element of $M(3)$ with line sum $L$ can be written as a sum of $L$ such matrices. The non-uniqueness of this sum is captured entirely by the dependence relation (or syzygy)
\begin{equation}J\ \ = \ \ P_e + P_{(123)}+P_{(132)}\ \  =\ \ P_{(13)}+P_{(12)}+ P_{(23)}.\end{equation}
Thus, if 
$$M=\sum\limits_{t=1}^6 a_t P_t,$$
then $M$ may be identified with an element of $\mathbb{Z}^6$ with each $a_t\ge 0$ and $\sum\limits_{t=1}^6 a_t = \rho(M).$ The dependence relation requires that, for all $t\ge 0$, 
$$(a_1+t, a_2+t, a_3+t, a_4, a_5, a_6)\ \ \text{and}\ \  \ (a_1, a_2, a_3, a_4 + t, a_5+t, a_6 +t)$$    
represent the same semi-magic square.  Thus we see that $M$ has a unique representative as a sextuple when $a_t=a_u=0$ with at least one $t$ in  $\{1, 2, 3\}$ and at least one $u$ in $\{4, 5, 6\}.$
In general, the following definition determines the convention for uniqueness in this work.

\begin{definition}  Suppose $M$ is represented by a sextuple ${\bf a} = (a_1, a_2, a_3, a_4, a_5, a_6).$ We say that ${\bf a}$ is the  {\bf upshifted} representative for $M$ if some $a_t=0$ for $t$ in $\{4, 5, 6\}$.
\end{definition}

\begin{definition}  Let ${\bf j} = (1, 1, 1, 0, 0, 0).$  Suppose 
$${\bf a} = (m_0, m_0, m_0, 0, 0, 0) + (b_1, b_2, b_3, a_4, a_5, a_6) = m_0{\bf j} + {\bf a}_{red}$$
is the upshifted representative for $M$ and at least one $b_t=0$.  We call ${\bf a}_{red}$ the {\bf reduced (or holey) part} of ${\bf a},$ with corresponding semi-magic square
$M_{red} = M - m_0J.$
\end{definition}
Then $M_{red}$ is uniquely represented by a sextuple, and we have 
$$\rho(M) = 3m_0 + \rho(M_{red}).$$ 
Note that $M=M_{red}$ if and only if at least one entry of $M$ equals zero.

\begin{example}  Suppose ${\bf a} = (2, 3, 4, 0, 1, 0).$  Then 
$$M = 2 P_1 + 3 P_2 + 4P_3 + P_5 =  \begin{bmatrix} 2 & 5 & 3 \\ 4 & 2 & 4 \\ 4 & 3 & 3 \end{bmatrix}$$
and 
$$\rho(M) = 2 + 3 + 4 + 1 = 10.$$
Since  ${\bf a}  = (2, 2, 2, 0, 0, 0) + (0, 1, 2, 0, 1, 0),$
$$M= 2J + M_{red} =  \begin{bmatrix} 2 & 2 & 2 \\ 2 & 2 & 2 \\ 2 & 2 & 2 \end{bmatrix} +   \begin{bmatrix} 0 & 3 & 1 \\ 2 & 0 & 2 \\ 2 & 1 & 1 \end{bmatrix}$$

and  
$$\rho(M) = \rho(2J) + \rho(M_{red}) = 6 + 4 = 10.$$
\end{example}

Finally, we consider the definitions of minimum and  maximum of $M$ with respect to the sextuple representative.  The dictionary (7.4)  shows the sums as  matrix entries.

\begin{proposition}
Suppose $M$ is represented by the sextuple ${\bf a}.$  Then the minimum and maximum entries of $M$ are given by 
$$\min(M) = \min(a_1, a_2, a_3) + \min(a_4, a_5, a_6).$$
and
$$\max(M) = \max(a_1, a_2, a_3) + \max(a_4, a_5, a_6).$$
These values are independent of the representative ${\bf a}$ used for $M$.
\end{proposition}

\begin{proof}  Note that, for instance,  $P_1$ shares a common entry of 1 in distinct positions with respect to $P_4$, $P_5$, and $P_6,$  and shares no common entry of 1 in any position with respect to $P_2$ and $P_3.$  Thus, in any linear combination of the $P_i$,  the entries of $M$ are just the pairwise sums of elements $a_i$, with one from each set listed.  The assertions follow immediately.  
\end{proof}

\section{Symmetries and orbits}

We now consider symmetries of $M(3)$ that preserve line sums \cite{RVR}.  The usual symmetries of determinant, generated by row switches, column switches, and transpose, preserve both the semi-magic property and  line sums.  The group of  linear symmetries preserving line sum is generated by such motions and is isomorphic to the  wreath product $G = S_3 \wr \mathbb{Z}/2,$  where $S_3$ is the symmetric group on three letters.

In particular, $|G|=72,$ and every element factors uniquely as $g=R(\sigma)C(\tau)T^\epsilon$ with $\sigma, \tau$ in $S_3$ and $\epsilon$ in $\{0, 1\}$; here $R(\sigma)$ denotes a permutation of rows, $C(\tau)$ a permutation of columns, and $T$ the operation of matrix transpose. Then
$$g\cdot M = P_\sigma M P^{-1}_\tau\ \ (\epsilon=0)\quad\text{or}\quad g\cdot M = P_\sigma M^T P_\tau^{-1}\ \ (\epsilon=1).$$ 

 It will be convenient to further emphasize the wreath product action by the identification
\begin{equation}M\quad \leftrightarrow\quad {\bf a} = (a_1, a_2, a_3, a_4, a_5, a_6) \quad \leftrightarrow \quad 
\begin{ytableau}
\ a_1 & a_2 & a_3\\
\ a_4 &  a_5 & a_6
\end{ytableau}.\end{equation}
Here the single relation (3.1) takes the form
\begin{equation}\begin{ytableau}
\  1 & 1 & 1\\
\ 0 &  0 & 0
\end{ytableau} = \begin{ytableau}
\ 0 & 0 & 0\\
\ 1 &  1 & 1
\end{ytableau}.\end{equation}

On the rectangle, the wreath product action permutes rows and allows any permutation of entries in a given row; thus, $G$ is isomorphic to the subgroup of the symmetric group $S_6$ generated by the three involutions
$$\langle(12),\ (23),\ (14)(25)(36)\rangle.$$ We leave it to the reader to determine the precise correspondence with matrix operations; these can be deduced readily from the examples in Section 6 and the dictionary (7.4).  

If the sextuple is upshifted, then each orbit may be further represented as
$$\begin{ytableau}
\ b_1 & b_2 & b_3\\
\ b_4 &  b_5 & 0
\end{ytableau}$$
such that the filling is non-increasing along rows.   With respect to the decomposition in Section 3,  we have, using upshifted forms,
\begin{equation}M \ =\  m_0 J \ +\ M_{red}  \quad \to \quad 
\begin{ytableau}
\ m_0 & m_0 & m_0\\
\ 0 &   0 &  0
\end{ytableau}\quad 
+\quad
\begin{ytableau}
\ a_1 & a_2 & a_3\\
\ a_4 &  a_5 & a_6
\end{ytableau},
\end{equation}
where the latter term has a zero in each row.
Since the first term, corresponding to $m_0J$, is invariant under the group action, orbit characteristics are entirely determined by the rectangle of the reduced part, which is upshifted by definition. That is, orbit behavior corresponds to the possible rectangles

$$\begin{ytableau}
\ b_1 & b_2 & 0\\
\ b_4 &  b_5 & 0
\end{ytableau}$$
with non-increasing row elements.  As a convention for orbits, we also order rows with non-decreasing number of zeros.  It is illustrative to note that the following distinct representatives are in the same orbit, but the second is the preferred for denoting the orbit: 
$$\begin{ytableau}
\ 1 & 1 & 1\\
\ 1 & 1 & 0
\end{ytableau},\ \ \ \begin{ytableau}
\ 2 & 2 & 1\\
\ 0 &  0 & 0
\end{ytableau}.
$$

In general, Table 1 organizes by orbit representative, size of orbit $o_M$, and stabilizer subgroup.  The dihedral group $D_{2n}$ is the group of symmetries of a regular $n$-sided polygon. For examples of rectangles with corresponding semi-magic square, all orbit types appear in Section 6 at least once.
\vspace{10pt}

\begin{table}
\caption{Orbit sizes for reduced forms (\cite{Do3})}
\label{tab:1}
\begin{tabular}{p{6cm}|p{2cm}|p{3.5cm}}
\hline\noalign{\smallskip}
  &  &  \\
Orbit  Representative $M$

(distinct $a, b, c >0;\ d >0$)
& $|o_M|$ & Stabilizer Subgroup

(first entry) \\
\noalign{\smallskip}\hline\noalign{\smallskip}  

$\begin{ytableau}
\ 0 & 0 & 0\\
\ 0 & 0 & 0
\end{ytableau}$

 & 1 & G\\   \noalign{\smallskip}\hline\noalign{\smallskip}  

$\begin{ytableau}
\ a & a & 0\\
\ 0 &  0 & 0
\end{ytableau}, \ \ \ 
\begin{ytableau}
\ a & 0 & 0\\
\ 0 &  0 & 0
\end{ytableau}$

  & 6 & $\langle(12), (45), (56)\rangle \cong D_{12}$ \\  \noalign{\smallskip}\hline\noalign{\smallskip}  

$\begin{ytableau}
\ a & a & 0\\
\ a &  a & 0
\end{ytableau},\ \ \ 
$\begin{ytableau}
\ a & 0 & 0\\
\ a &  0 & 0
\end{ytableau}$
$

& 9 &  $\langle (12), (45), (14)(25)(36) \rangle$ 

\quad $\cong D_8$

\\   \noalign{\smallskip}\hline\noalign{\smallskip}  

$
\begin{ytableau}
\ a & b & 0\\
\ 0 &  0 & 0
\end{ytableau}$

& 12 &  $\langle(45), (56)\rangle\cong D_6 \cong S_3$ \\   \noalign{\smallskip}\hline\noalign{\smallskip}  

$\begin{ytableau}
\ a & a & 0\\
\ b &  b & 0
\end{ytableau},\ \ \ 
\begin{ytableau}
\ a & a & 0\\
\ d &  0 & 0
\end{ytableau}, \ \ \
\begin{ytableau}
\ a & 0 & 0\\
\ b &  0 & 0
\end{ytableau}$    

& 18 &  $\langle (12), (45)\rangle$ 

\quad $\cong \mathbb{Z}/2 \times \mathbb{Z}/2$ \\   \noalign{\smallskip}\hline\noalign{\smallskip}  

$ 
 \begin{ytableau}
\ a & a & 0\\
\ a &  b & 0
\end{ytableau},\ \ \ 
\begin{ytableau}
\ a & a & 0\\
\ b &  c & 0
\end{ytableau},\ \ \ 
\begin{ytableau}
\ a & b & 0\\
\ d &  0 & 0
\end{ytableau}$

   & 36 &   $\langle (12)\rangle\cong \mathbb{Z}/2$ \\    \noalign{\smallskip}\hline\noalign{\smallskip}  

$\begin{ytableau}
\  a & b & 0\\
\ a &  b & 0
\end{ytableau}$

  & 36 &  $\langle(14)(25)(36)\rangle \cong \mathbb{Z}/2$ \\    \noalign{\smallskip}\hline\noalign{\smallskip}  

All remaining cases    (no symmetry)

& 72  &    \{e\} \\    
  \noalign{\smallskip}\hline\noalign{\smallskip}
\end{tabular}
\end{table}

Next we clarify properties needed to work with $M(3, s).$  As noted in Section 3, the maximum value of $M$ is given by the sum of the maximum values in each row of the rectangle.

\begin{proposition}  Suppose ${\bf a}$ is the upshifted representative for $M$ in $M(3, s)$.  

Define $m_1=\max(a_4, a_5, a_6)$ and $s_1 = s-m_1.$  Then $M^*$ is represented by the upshifted rectangle
\begin{equation}
\begin{ytableau}
\ s_1 & s_1 & s_1\\
\ m_1 &  m_1 & m_1
\end{ytableau}\ \ - \ \ 
\begin{ytableau}
\ a_1 & a_2 & a_3\\
\ a_4 &  a_5 & a_6
\end{ytableau}\ .
\end{equation}
\end{proposition}

Effectively, to find the rectangle for $M^*$,  one finds the complement of  the rectangle for $M$ with respect to
$$
\hat{1} = sJ =  \begin{ytableau}
\ s & s & s\\
\ 0 &  0 & 0
\end{ytableau}\ .$$
This complement includes terms in the second row that complete to a multiple of $(1, 1, 1),$ which in turn transfers to the first row. That is, the sum of $M$ and $M^*$ should equal $sJ$. 

For example, if  $M$ in $M(3, 7)$  corresponds to 
\begin{equation} \begin{ytableau}
\ 3 & 1 & 1\\
\ 2 &  1 & 0
\end{ytableau}\ ,$$
then
$$\begin{ytableau}
\ 3 & 1 & 1\\
\ 2 &  1 & 0
\end{ytableau}\  +\ 
 \begin{ytableau}
\ 2 & 4 & 4\\
\ 0 &  1 & 2
\end{ytableau}\ = \ 
 \begin{ytableau}
\ 5 & 5 & 5\\
\ 2 &  2 & 2
\end{ytableau}\ =\ 
 \begin{ytableau}
\ 7 & 7 & 7\\
\ 0 &  0 & 0
\end{ytableau}\ .
\end{equation}
\vspace{10pt}

The following proposition is straightforward.

\begin{proposition} As a  finite graded poset, $M(3, s)$ is preserved by $G$.  If $M$ and $N$ belong to the same orbit under $G$, then $M^*$ and $N^*$ also belong to the same orbit under $G$.  The orbits for $M$ and $M^*$ contain the same number of elements.
\end{proposition}

See Figures 1 and 2 for the poset $M(3, 1)$ using semi-magic squares and rectangles, respectively.

\begin{figure}
{\tiny \begin{tikzpicture}
  \node (r0) at (0,0) {$\begin{bmatrix}
 0 & 0 & 0\\
 0 &  0 & 0\\
 0 & 0 & 0
\end{bmatrix}$};
  \node (r11) at (-4,1.7){$\begin{bmatrix}
 1 & 0 & 0\\
 0 &  1 & 0\\
 0 & 0 & 1
\end{bmatrix}$};
  \node (r12) at (-2.4, 1.7){$\begin{bmatrix}
 0 & 0 & 1\\
 1 &  0 & 0\\
 0 & 1 & 0
\end{bmatrix}$};  
  \node (r13) at (-.8,1.7) {$\begin{bmatrix}
 0 & 1 & 0\\
 0 &  0 & 1\\
 1 & 0 & 0
\end{bmatrix}$};
\node (r14) at (.8,1.7) {$\begin{bmatrix}
 0 & 0 & 1\\
 0 &  1 & 0\\
 1 & 0 & 0
\end{bmatrix}$};
\node (r15) at (2.4,1.7) {$\begin{bmatrix}
 0 & 1 & 0\\
 1 &  0 & 0\\
 0 & 0 & 1
\end{bmatrix}$};
 \node (r16) at (4,1.7){$\begin{bmatrix}
 1 & 0 & 0\\
 0 &  0 & 1\\
 0 & 1 & 0
\end{bmatrix}$};
  \node (r21) at (-4,3.4) {$\begin{bmatrix}
 1 & 0 & 1\\
 1 &  1 & 0\\
 0 & 1 & 1
\end{bmatrix}$};
  \node (r22) at (-2.4, 3.4) {$\begin{bmatrix}
 1 & 1 & 0\\
 0 &  1 & 1\\
 1 & 0 & 1
\end{bmatrix}$};  
  \node (r23) at (-.8,3.4) {$\begin{bmatrix}
 0 & 1 & 1\\
 1 &  0 & 1\\
 1 & 1 & 0
\end{bmatrix}$};
\node (r24) at (.8,3.4) {$\begin{bmatrix}
 0 & 1 & 1\\
 1 &  1 & 0\\
 1 & 0 & 1
\end{bmatrix}$};
\node (r25) at (2.4,3.4) {$\begin{bmatrix}
 1 & 0 & 1\\
 0 &  1 & 1\\
 1 & 1 & 0
\end{bmatrix}$};
 \node (r26) at (4,3.4) {$\begin{bmatrix}
 1 & 1 & 0\\
 1 &  0 & 1\\
 0 & 1 & 1
\end{bmatrix}$};
  \node (r3) at (0,5.1) {$\begin{bmatrix}
 1 & 1 & 1\\
 1 &  1 & 1\\
 1 & 1 & 1
\end{bmatrix}$};
  \draw  (r0) -- (r11)  (r0) -- (r12)  (r0) -- (r13) 
  (r0) -- (r14)  (r0) -- (r15) (r0) -- (r16)
  (r11) -- (r21) (r11) -- (r22) (r12) -- (r21) 
  (r12) -- (r23) (r13) -- (r22) (r13) -- (r23)
  (r14) -- (r24) (r14) -- (r25) (r15) -- (r24) (r15) -- (r26) (r16) -- (r25) (r16) -- (r26) 
  (r21) -- (r3) (r22) -- (r3)   (r23) -- (r3) (r24) -- (r3)   (r25) -- (r3) (r26) -- (r3) ;
  \end{tikzpicture}}
  \caption{The poset $M(3, 1)$ as semi-magic squares}
  \end{figure}
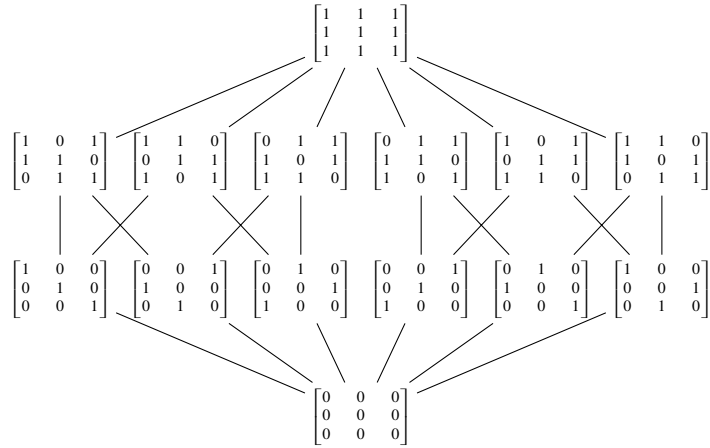
\begin{figure}
{\tiny
  \begin{tikzpicture}
  \node (r0) at (0,0) {$\ \ \begin{ytableau}
\  0 & 0 & 0\\
\ 0 &  0 & 0
\end{ytableau}_{\ 1}$\ \ };
  \node (r11) at (-4,1.7) {\ \ $\begin{ytableau}
\  1 & 0 & 0\\
\ 0 &  0 & 0
\end{ytableau}_{\ 1}$ \ \ };
  \node (r12) at (-2.4, 1.7) {\ \ $\begin{ytableau}
\  0 & 1 & 0\\
\ 0 &  0 & 0
\end{ytableau}_{\ 1 }$\ \ };  
  \node (r13) at (-.8,1.7) {\ \ $\begin{ytableau}
\  0 & 0 & 1\\
\ 0 &  0 & 0
\end{ytableau}_{\ 1}$\ \ };
\node (r14) at (.8,1.7) {\ \ $\begin{ytableau}
\  0 & 0 & 0\\
\ 1 &  0 & 0
\end{ytableau}_{\ 1}$\ \ };
\node (r15) at (2.4,1.7) {\ \ $\begin{ytableau}
\  0 & 0 & 0\\
\ 0 &  1 & 0
\end{ytableau}_{\ 1}$\ \ };
 \node (r16) at (4,1.7) {\ \ $\begin{ytableau}
\  0 & 0 & 0\\
\ 0 &  0 & 1
\end{ytableau}_{\ 1}$\ \ };
  \node (r21) at (-4,3.4) {\ \ $\begin{ytableau}
\  1 & 1 & 0\\
\ 0 &  0 & 0
\end{ytableau}_{\ 2}$ \ \ };
  \node (r22) at (-2.4, 3.4) {\ \ $\begin{ytableau}
\  1 & 0 & 1\\
\ 0 &  0 & 0
\end{ytableau}_{\ 2 }$\ \ };  
  \node (r23) at (-.8,3.4) {\ \ $\begin{ytableau}
\  0 & 1 & 1\\
\ 0 &  0 & 0
\end{ytableau}_{\ 2}$\ \ };
\node (r24) at (.8,3.4) {\ \ $\begin{ytableau}
\  0 & 0 & 0\\
\ 1 &  1 & 0
\end{ytableau}_{\ 2}$\ \ };
\node (r25) at (2.4,3.4) {\ \ $\begin{ytableau}
\  0 & 0 & 0\\
\ 1 &  0 & 1
\end{ytableau}_{\ 2}$\ \ };
 \node (r26) at (4,3.4) {\ \ $\begin{ytableau}
\  0 & 0 & 0\\
\ 0 &  1 & 1
\end{ytableau}_{\ 2}$\ \ };
  \node (r3) at (0,5.1) {$\ \ \begin{ytableau}
\  1 & 1 & 1\\
\ 0 &  0 & 0
\end{ytableau}_{\ 12}$\ \ };
  \draw  (r0) -- (r11)  (r0) -- (r12)  (r0) -- (r13) 
  (r0) -- (r14)  (r0) -- (r15) (r0) -- (r16)
  (r11) -- (r21) (r11) -- (r22) (r12) -- (r21) 
  (r12) -- (r23) (r13) -- (r22) (r13) -- (r23)
  (r14) -- (r24) (r14) -- (r25) (r15) -- (r24) (r15) -- (r26) (r16) -- (r25) (r16) -- (r26) 
  (r21) -- (r3) (r22) -- (r3)   (r23) -- (r3) (r24) -- (r3)   (r25) -- (r3) (r26) -- (r3) ;
  \end{tikzpicture}}
  \caption{The poset $M(3, 1)$ as rectangles, with path numbers}
  \end{figure}
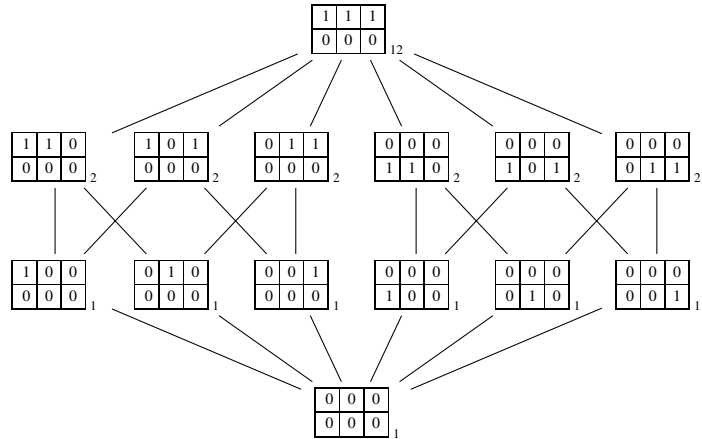

\section{Lattice path enumeration}

With the diagrammatic notation in place, we now consider lattice path enumeration in $M(3).$   Per the usual conventions, multinomial coefficients vanish when the lower index is negative or exceeds the upper index.  For clarity, all parts of the upper index are listed in the lower indices of a multinomial coefficient.

\begin{theorem}  Suppose ${\bf a}$ is the upshifted representative for $M$ in $M(3)$ with $m_0=min(M)$. Then the number of lattice paths in $M(3)$ from $\hat{0}$ to $M$ is given by the path number
\begin{equation}v(M)\ =\ \sum\limits_{t=0}^{m_0} \begin{pmatrix} \rho(M) \\  a_1-t,\ a_2-t,\ a_3-t,\ a_4+t,\ a_5+t,\ a_6+t \end{pmatrix}.\end{equation}
Additionally, $v(M)$ is constant on the orbit for $M$ with respect to  $G$.
\end{theorem}

\begin{proof} First note that any such path represents a stepwise construction of $M$ by repeatedly adding permutation matrices to the zero matrix. That is, each path from $\hat{0}$ to $M$ is given by a word composed of letters in $\{P_1, \dots, P_6\}$ and of length $\rho(M).$   On the other hand,  each of these words must yield a representative of $M$. There are $m_0+1$ such representatives of the form 
$$(a_1-t,\ a_2-t,\ a_3-t,\ a_4+t,\ a_5+t,\ a_6+t),$$
for which the respective number of paths is given by the corresponding multinomial coefficient.  A given path arises from only one of these forms and then uniquely.

The orbit property follows immediately from the formula.
\end{proof}
In general, we use the upshifted representative for clear indexing.  Since the range of the index covers the nonzero summands, we may extend the sum over all integers.  If we use another representative for $M$, the non-zero summands are merely reindexed from the upshifted case.

Now, because any such word of length $t$ represents a unique path in $M(3)$, we have the following analogue to the row sum formula for Pascal's triangle.

\begin{corollary}  Fix $t\ge 0.$  Then
$$6^t = \sum\limits_{\rho(M)=t} v(M).$$
\end{corollary}
 
Next, by expanding the multinomial coefficients and factoring, we have
\begin{corollary}  Suppose ${\bf a}$ is the upshifted representative for $M$ with $m_0=\min(M).$  Then the number of paths in $M(3)$ from $\hat{0}$ to $M$ equals
\begin{equation}v(M)\ =\ \begin{pmatrix}  \rho(M)\\ a_1,\ a_2,\ a_3,\ a_4,\ a_5,\ a_6\end{pmatrix} \cdot \sum\limits_t q({\bf a}, t),\end{equation} 
where 
\begin{align}
q({\bf a}, t) &= \bigg [\begin{pmatrix} a_1 \\ t\end{pmatrix}\begin{pmatrix} a_2 \\ t\end{pmatrix}\begin{pmatrix} a_3 \\  t\end{pmatrix}\bigg]/ \bigg[\begin{pmatrix} a_4+t \\ t\end{pmatrix}\begin{pmatrix} a_5+t \\ t\end{pmatrix}\begin{pmatrix} a_6+t \\  t\end{pmatrix}\bigg]\\[5pt]
\notag & =  \frac{(a_1)_t\ (a_2)_t\ (a_3)_t}{(a_4+t)_{t} (a_5+t)_{t} (a_6+t)_{t}}
\end{align}

\noindent and $(a)_t = a(a+1)\dots (a+t-1)$ denotes the shifted  factorial.
\end{corollary}
Since one of $a_4, a_5, a_6$ equals zero, the sum corresponds to a terminating hypergeometric series of type $\,_3F_2$ evaluated at $z=-1.$ We return to this formulation in Section 7. 
\vspace{5pt}

Now consider the  example $M=2J$.  Since ${\bf a} = (2, 2, 2, 0, 0, 0)$,
$$v(2J)\ = \ \sum\limits_t\ \begin{pmatrix} 6 \\  2-t,\ 2-t,\ 2-t,\ t,\ t,\ t \end{pmatrix}\  =\ 90 + 720 + 90\ =\ 900.$$

In general, we have
\begin{corollary}  Suppose $M=sJ.$  Then the number of paths in $M(3)$ from $\hat{0}$ to $M$ equals
\begin{equation}p(s)\ =\ v(sJ)\ =\ \begin{pmatrix} 3s \\  s,\ s, \ s \end{pmatrix}\sum\limits_{t=0}^s \begin{pmatrix} s \\  t\end{pmatrix}^3.\end{equation}
\end{corollary}

The sequence $p(s)$  occurs as A306642 in the OEIS \cite{Slo}, and the factors 
$$F(s) = \sum\limits_{t=0}^s \begin{pmatrix} s \\  t\end{pmatrix}^3$$
are the Franel numbers (A000172).   
\vspace{10pt}

\begin{tabular}{ | p{1.5 cm} | p{.9 cm}  | p{.9 cm} | p{.9 cm} |  p{1 cm} | p{1.7 cm} | p{2 cm}|  }
\hline
	 $n$ &  0 & 1 & 2 & 3 & 4 & 5  \\ \hline
	A000172 & 1  &    2  & 10     & 56      & 346  & 2252 \\ \hline
	A306642 & 1 &   12 & 900   & 94080   & 11988900  & 1704214512 \\
\hline
\end{tabular}
\vspace{10pt}

Now $F(s)$ satisfies the recurrence (\cite{Fr1})
$$(s+1)^2F(s+1)\ =\ (7s^2+7s+2) F(s)  \ + \ 8s^2 F(s-1) ,$$
from which we immediately derive
\begin{corollary}  For $s>0$, $p(s)$ satisfies the three-term recurrence
\begin{eqnarray}
 \notag (s+1)^4p(s+1) &=& 3(3s+2)(3s+1)(7s^2+7s+2)p(s)\\
 &\ & \qquad\qquad\qquad + \quad \ 72(9s^2-4)(9s^2-1)p(s-1).
\end{eqnarray}
\end{corollary}

Furthermore, the recurrences imply the quotient limits
\begin{equation}\lim\limits_{s\to\infty}\ \frac{F(s)}{F(s-1)} = 8\qquad \text{and}\qquad\lim\limits_{s\to\infty}\ \frac{p(s)}{p(s-1)} = 216.\end{equation}

\section{Vandermonde convolution for finite graded posets}

As a first application of the formula, we consider a version of Vandermonde convolution for the finite graded posets $M(3, s).$   The path numbers give a set of invariants for the poset; these numbers are restricted as solutions to a family of equations, one for each rank of the poset.

In general, let $P$ be a finite graded poset with $\hat{0}$ and $\hat{1}$.  It will be convenient here to assume that $P$ is self-dual, but this assumption is not necessary in general.  For $x$ in $P$, let $v_P(x)$ be the path number of $x$; that is, $v_P(x)$ is the number of directed paths from $\hat{0}$ to $x$ in $P$.  

If we fix a rank level $P_k$ in $P$, then every path from $\hat{0}$ to $\hat{1}$ intersects uniquely with an element of $P_k$, which in turn gives a partition of all such paths, indexed by elements of $P_k.$  

\begin{theorem}[Vandermonde convolution for finite graded posets]  Suppose $P$ has rank $n$.  Let $\rho(x)$ be the rank of $x$ in $P,$ and assume that $P$ is self-dual with involution $x \mapsto x^*.$ Then, for $0\le k \le \rho(\hat{1}),$
\begin{equation}v_P(\hat{1}) = \sum\limits_{x\in P_k}\  v_P(x)v_P(x^*).\end{equation}
\end{theorem}

\begin{proof}  Since $P$ is self-dual,  the number of paths from $x$ to $\hat{1}$ equals the number of paths from $\hat{0}$ to $x^*.$  
\end{proof}

To visualize using the diagram for $P$, one implements this equation as a dot product of path numbers between dual rank levels.

Applied to $M(3, s),$ we have
\begin{corollary}  Fix $s\ge 0,$ and let $P=M(3, s).$  Then, for $0 \le k \le 3s$,
\begin{equation}v(sJ) =  \sum\limits_{\substack{\rho(M)=k\\ max(M) \le s}}\  v(M)v(sJ-M).\end{equation}
\end{corollary}

\begin{example}{(Figure 3)}  The poset $M(3, 2)$ is represented in both rectangle and matrix form below.  We orient the diagrams so that rank levels are given as columns with increasing rank.  For the sake of presentation, all elements in a given orbit are listed by a single representative rectangle with subscript denoting the orbit size and common path number.

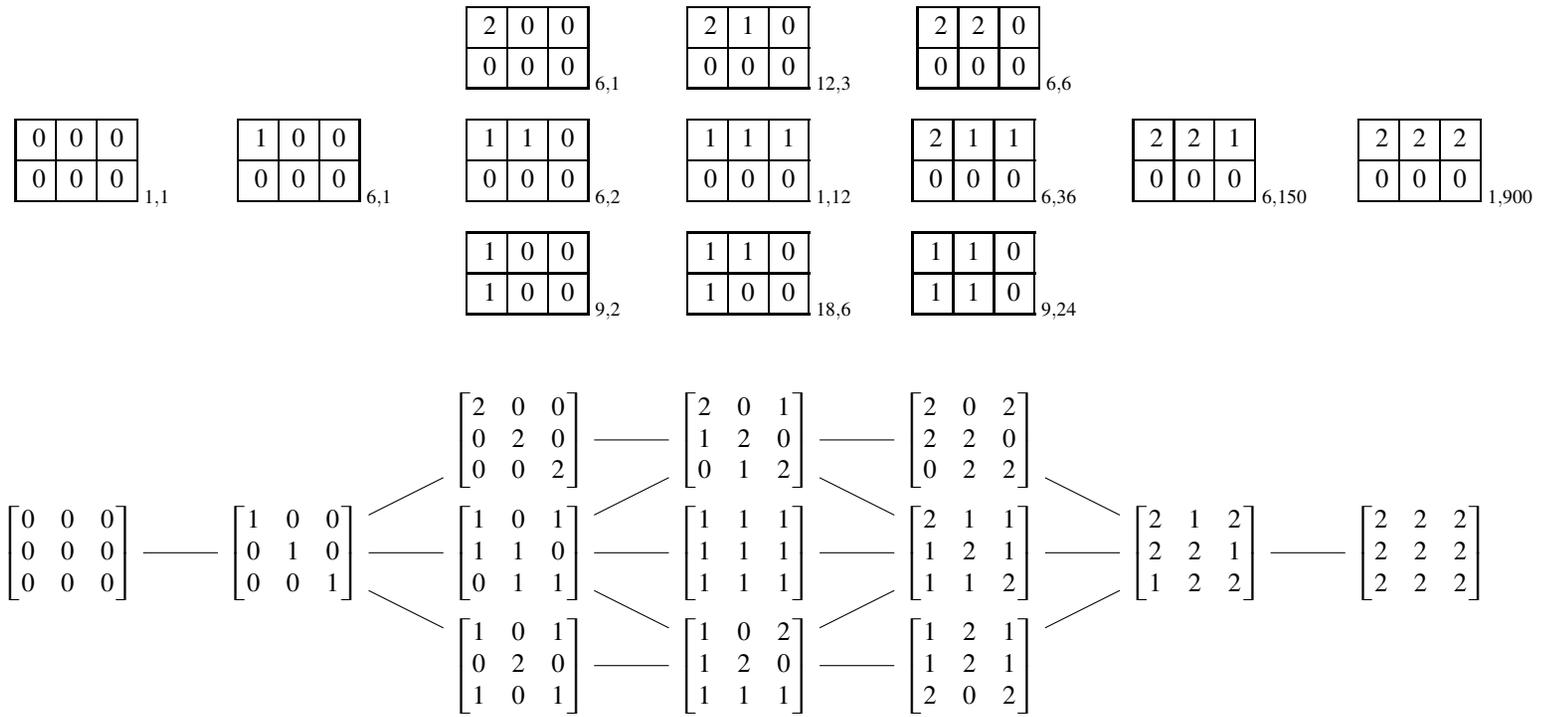
\begin{sidewaysfigure}
{\vspace{350pt}
\begin{tikzpicture}
  \node (min) at (0,0) {$\ \ \begin{ytableau}
\  0 & 0 & 0\\
\ 0 &  0 & 0
\end{ytableau}_{\ 1, 1}$\ \ };
  \node (a) at (3,0) {\ \ $\begin{ytableau}
\  1 & 0 & 0\\
\ 0 &  0 & 0
\end{ytableau}_{\ 6, 1}$ \ \ };
  \node (b) at (6, 0) {\ \ $\begin{ytableau}
\  1 & 1 & 0\\
\ 0 &  0 & 0
\end{ytableau}_{\ 6, 2}$\ \ };  
  \node (c) at (6,1.5) {\ \ $\begin{ytableau}
\  2 & 0 & 0\\
\ 0 &  0 & 0
\end{ytableau}_{\ 6, 1}$\ \ };
  \node (d) at (6,-1.5) {\ \ $\begin{ytableau}
\  1 & 0 & 0\\
\ 1 &  0 & 0
\end{ytableau}_{\ 9, 2}$\ \ };
 \node (e) at (9,0) {\ \ $\begin{ytableau}
\  1 & 1 & 1\\
\ 0 &  0 & 0
\end{ytableau}_{\ 1, 12}$\ \ };
  \node (f) at (9,1.5) {\ \ $\begin{ytableau}
\  2 & 1 & 0\\
\ 0 &  0 & 0
\end{ytableau}_{\ 12, 3}$\ \ };
  \node (g) at (9,-1.5) {\ \ $\begin{ytableau}
\  1 & 1 & 0\\
\ 1 &  0 & 0
\end{ytableau}_{\ 18, 6}$\ \ };
  \node (h) at (12,0) {\ \ $\begin{ytableau}
\  2 & 1 & 1\\
\ 0 &  0 & 0
\end{ytableau}_{\ 6, 36}$\ \ };
  \node (i) at (12,1.5) {\ \ $\begin{ytableau}
\  2 & 2 & 0\\
\ 0 &  0 & 0
\end{ytableau}_{\ 6, 6}$\ \ };
  \node (max) at (12,-1.5) {\ \ $\begin{ytableau}
\  1 & 1 & 0\\
\ 1 &  1 & 0
\end{ytableau}_{\ 9, 24}$\ \ };
  \node (max) at (15,0) {\ \ $\begin{ytableau}
\  2 & 2 & 1\\
\ 0 &  0 & 0
\end{ytableau}_{\ 6, 150}$\ \ };
  \node (max) at (18,0) {\ \ $\begin{ytableau}
\  2 & 2 & 2\\
\ 0 &  0 & 0
\end{ytableau}_{\ 1, 900}$\ \ };
  \end{tikzpicture}\vspace{20pt}
\begin{tikzpicture}
  \node (min) at (0,0) {$\begin{bmatrix}
 0 & 0 & 0\\
 0 &  0 & 0\\
 0 & 0 & 0
\end{bmatrix}$};
  \node (a1) at (3,0) {$\begin{bmatrix}
1 & 0 & 0\\
0 &  1 & 0\\
0 & 0 & 1
\end{bmatrix}$};
  \node (b1) at (6, 0) {$\begin{bmatrix}
1 & 0 & 1\\
1 &  1 & 0\\
0 & 1 & 1
\end{bmatrix}$};  
  \node (b2) at (6,1.5) {$\begin{bmatrix}
2 & 0 & 0\\
0 &  2 & 0\\
0 & 0 & 2
\end{bmatrix}$};
  \node (b3) at (6,-1.5) {$\begin{bmatrix}
1 & 0 & 1\\
0 &  2 & 0\\
1 & 0 & 1
\end{bmatrix}$};
 \node (c1) at (9,0) {$\begin{bmatrix}
1 & 1 & 1\\
1 &  1 & 1\\
1  &  1 & 1
\end{bmatrix}$};
  \node (c2) at (9,1.5) {$\begin{bmatrix}
 2 & 0 & 1\\
1 &  2 & 0\\
0  &  1 &  2
\end{bmatrix}$};
  \node (c3) at (9, -1.5) {$\begin{bmatrix}
1 &  0 & 2\\
1 &  2 & 0\\
1  & 1 & 1
\end{bmatrix}$ };
  \node (d1) at (12,0) {$\begin{bmatrix}
2 & 1 & 1\\
1 &  2 & 1\\
1 &  1 & 2
\end{bmatrix}$};
  \node (d2) at (12,1.5) {$\begin{bmatrix}
2 & 0 & 2\\
2 &  2 & 0\\
0  & 2 & 2
\end{bmatrix}$};
  \node (d3) at (12,-1.5) {$\begin{bmatrix}
1 & 2 & 1\\
1 &  2 & 1\\
2 & 0 & 2
\end{bmatrix}$};
  \node (e1) at (15,0) {$\begin{bmatrix}
 2 & 1 & 2\\
2 &  2 & 1\\
1 & 2  & 2
\end{bmatrix}$};
\node (max) at (18, 0) {$\begin{bmatrix}
2 & 2 & 2\\
2 &  2 & 2\\
2  & 2 & 2
\end{bmatrix}$};
  \draw  (min) -- (a1)  (a1) -- (b1)  (a1) -- (b2) 
  (a1) -- (b3)  (b1) -- (c1) (b2) -- (c2) (b1) -- (c2) (b1) -- (c3)
  (b3) -- (c3) (c1) -- (d1) (c2) -- (d1) (c2) -- (d2) (c3) -- (d1)
  (c3) -- (d3) (d1) -- (e1) (d2) -- (e1) (d3) -- (e1) (e1) -- (max);
  \end{tikzpicture}
  \vspace{20pt}
  }
  \caption{The poset of orbits in $M(3, 2)$ as rectangles (top) and semi-magic squares (bottom)}
  \end{sidewaysfigure}
  \vspace{5pt}

Since $v(M)$ is constant on orbits for $G$, the convolution formula may be simplified as
$$v(sJ) =  \sum\limits_{M}\  o_M\ v(M)v(sJ-M),$$
where $o_M$ denotes the size of the orbit for $M$, and the sum is now over a set or representatives for each orbit in $M(3, s)_k$.  For each rank pair in $M(3, 2)$, we obtain

\begin{eqnarray}
 Ranks \  \ \  0, \ 6 &:& \ \ 900\\
\notag Ranks \ \ \  1,\ 5 &:& \ \ 900 \ =  \ 6 \cdot 1 \cdot 150\\
\notag Ranks \ \ \  2,\ 4 &:&  \ \  900 \ = \ 6 \cdot 1\cdot 6\ \ \ \ \ +\ \ \  6\cdot 2\cdot 36\ \ \ +\ 9 \cdot 2 \cdot 24\\
\notag Rank \ \ \quad\ \  3 &:& \ \  900 \ =  \ 12 \cdot 3 \cdot 3\ \ \  +\ \ \ 1 \cdot 12\cdot 12\ +\ 18 \cdot 6 \cdot 6.
\end{eqnarray}
\end{example}

For larger $s$,  orbit count data for reduced squares is found in Table 1 of \cite{Do3}.  Both orbit types for size 36 and the type for size 72 first appear in $M(3, 3),$ $M(3, 4)$, and $M(3, 5),$ respectively.  For instance, consider the examples

{\small
$$
\begin{ytableau}
\  2 & 1 & 0\\
\ 1 &  0 & 0
\end{ytableau}\ \to \ 
 \begin{bmatrix}
 2 & 0 & 2\\
1 &  3 & 0\\
1 & 1  & 2
\end{bmatrix},\quad
 \begin{ytableau}
\  2 & 1 & 0\\
\ 2 &  1 & 0
\end{ytableau} \ \to \ 
 \begin{bmatrix}
2 & 1 & 3\\
2 &  4 & 0\\
2 & 1 & 3
\end{bmatrix}, \quad
\begin{ytableau}
\  3 & 1 & 0\\
\  2 &  1 & 0
\end{ytableau}\ \to\ 
\begin{bmatrix}
4 & 0 & 3\\
1 &  5 & 1\\
2  & 2 & 3
\end{bmatrix}.
$$ 
}

\section{Clebsch-Gordan coefficients and Regge Symmetries}

We recall the representation-theoretic formulation and indexing from \cite{Do1}.  For the Lie algebra $\frak{sl}(2, \mathbb{C})$, let $V(N)$ denote the irreducible representation space with  highest weight $N\ge 0.$ Variable names below reflect indexing in the Clebsch-Gordan decomposition and define the un-normalized Clebsch-Gordan coefficient $c_{m, n, k}(i, j).$  With suitably chosen highest weight vectors $\phi_N$,

$$V(m+n-2k) \subseteq V(m)\otimes V(n),$$
\begin{equation}f^{t-k}\ \phi_{m+n-2k}\ \ =\ \ \sum\limits_{i+j=t} \ c_{m,n,k}(i, j)\ \  f^i\phi_m\otimes f^j\phi_n.\end{equation}
Note that the weight of either side of the equality is $m+n-2t.$ In \cite{Do3}, we simply use the functional notation $C(M)$ for the coefficient.  See (7.4) below for a dictionary of indices, including 3-$j$ notation.

Returning to Corollary 5.3, we establish a functional connection between path numbers and the Clebsch-Gordan coefficients above. To each $M$ in $M(3)$, we assign a polynomial $F(M, z)$ and describe the behavior of $F$ when indices change according to the action of $G$. At $z=1$, we recover the invariance of $v(M)$ with respect to $G$; at $z=-1,$ we obtain an un-normalized analogue of the 3-$j$ symbol and an efficient derivation of the 72 Regge symmetries.

\begin{definition} Suppose ${\bf a}$ is the upshifted representative for $M$ in $M(3)$. Define

\begin{equation} F(M, z)\  \ = \ \sum\limits_{t=0}^{m_0} \begin{pmatrix} \rho(M) \\  a_1-t,\ a_2-t,\ a_3-t,\ a_4+t,\ a_5+t,\ a_6+t \end{pmatrix} z^t,\end{equation}
where $m_0\ = \ \min(a_1, a_2, a_3).$ 
\end{definition}
Alternatively, we have
\begin{align}
\notag F(M, z) \ &= \ \begin{pmatrix} \rho(M) \\ a_1, a_2, a_3, a_4, a_5, a_6 \end{pmatrix}\sum_{t} \frac{(a_1)_t (a_2)_t (a_3)_t}{(a_4+t)_t(a_5+t)_t(a_6+t)_t}\ z^t\\
& =\ \begin{pmatrix} \rho(M) \\ a_1, a_2, a_3, a_4, a_5, a_6 \end{pmatrix} {}_3F_2(-a_1, -a_2, -a_3; a_4+1, a_5+1, a_6+1; -z), 
\end{align}
where 
$(a)_t$ is the shifted factorial from Corollary 5.3.

It will be convenient to compare to the notation of (7.1);  we also include the 3-$j$ symbol notation, using the convention of \cite{Vi}, III.8, formula (9).  With  $m'=m+n-i-j-k$,
\begin{align}
\notag M \ &=\ \begin{bmatrix} a_1 +a_6 & a_3 + a_5 & a_2 + a_4\\ a_2 + a_5 & a_1 + a_4 & a_3 + a_6 \\ a_3 + a_4 & a_2 + a_6 & a_1 + a_5  \end{bmatrix}\\ 
&=\  \begin{bmatrix} n-k & m-k & k\\ i & j & m' \\ m-i & n-j & i+j-k  \end{bmatrix} \\  
\notag &=\  \begin{bmatrix} -j_1+j_2+j_3 & j_1-j_2+j_3 & j_1+j_2-j_3\\ j_1-m_1 & j_2-m_2 & j_3-m_3 \\ j_1+m_1 & j_2+m_2 & j_3+m_3  \end{bmatrix}.
\end{align}

There are many equivalent formulas for Clebsch-Gordan coefficients; we choose Theorem 9.2 of \cite{Do1}, with variables from the tensor product formulation:
\begin{equation}C(M) \ =\ \sum\limits_t (-1)^t  \begin{pmatrix} i+j-k \\ i-t \end{pmatrix}\begin{pmatrix} m-i \\ k-t \end{pmatrix}\begin{pmatrix} n-j \\ t \end{pmatrix}.\end{equation}

\begin{theorem}[Reciprocity for Clebsch-Gordan coefficients]  Suppose ${\bf a}$ is the upshifted representative for $M$ in $M(3)$ with $m_0=\min(a_1, a_2, a_3).$  Then 
$$F(M, 1)\ =\ v(M)$$
and 
\begin{equation}F(M, -1)\ = \  (-1)^{a_2+m_0}\ \begin{pmatrix} \rho(M) \\ a_1+a_5,\ a_2+a_6, \ a_3+a_4 \end{pmatrix}\ C(M).\end{equation}
\end{theorem}

\begin{proof} The first equation is Theorem 5.1.   For the second equation,  direct substitution gives
$$C(M)\ =\ \sum\limits_t (-1)^t  \begin{pmatrix} a_1 +a_5 \\ a_2+a_5-t \end{pmatrix}\begin{pmatrix} a_3+a_4 \\ a_2+a_4-t \end{pmatrix}\begin{pmatrix} a_2+a_6 \\ t \end{pmatrix}.$$
With the changes of variable $t \mapsto t + a_2$ followed by $t \mapsto -t$, one obtains the second equality by rearranging factorials. The second change of variable introduces the factor of $(-1)^{m_0}.$
\end{proof}

\noindent {\bf Remark.}  In \cite{LP}, the authors essentially identify the multinomial sum $F({\bf a}, -1)$ in the 3-$j$ notation of \cite{Vi}.  We describe using the notation of (7.1).  Random walks in the plane from $(0, 0)$ to $(m-2i, n-2j)$ are given using six possible steps:
$$\pm (1, -1),\quad \pm (1, 0), \quad \pm(0, 1),$$
with steps from each pair allowed $k$, $m-k$, and $n-k$ times, respectively.  All solutions to these conditions are of the form
\begin{align}
\notag(m-2i, n-2j) &= t(-1, 1)+(k-t)(1, -1)+(i-t)(-1, 0) + (m-k-i+t)(1, 0)\\
 \notag &\ \  \ \ +  (j-k+t)(0, -1) + (n-j-t)(0, 1). 
\end{align}
Of course, these five parameters determine a unique semi-magic square, and this random walk model is equivalent to the lattice path model of Section 5. 

Another combinatorial approach to $F(M, -1)$ using semi-magic squares appears in \cite{Lo1}, section 1.4, in particular (1.218), (1.221), and (11.261).  In this case, Clebsch-Gordan coefficients are derived directly from solid harmonics for $SU(2)$ using determinants and Rota's evaluation operation. Effectively, these methods enable a reworking of the theory of coupling of angular momentum in terms of semi-magic squares, bypassing techniques from Lie theory and differential equations.
\vspace{5pt}

We now consider how $F(M, z)$ transforms when indices are permuted by elements of $G$.  

\begin{proposition}  Suppose $g$ in $G$ sends $M$ to $M'$ and ${\bf a}$ to $g\cdot {\bf a}$.  If $g$ preserves each of the index sets $\{ 1, 2,3 \}$ and $\{ 4, 5, 6\}$,
then 
$$F(M, z) \ = \ F(M', z).$$
Otherwise
$$F(M, z) \ = \ z^{m_0}F(M', 1/z)$$
with $M'$ represented by the upshifted sextuple
$${\bf a'}\ =\ g\cdot{\bf a} + m_0(1,\ 1,\ 1, -1, -1, -1).$$
\end{proposition}

\begin{proof}  Using Definition 7.1, the first assertion is clear.   On the other hand, interchanging the noted index sets reverses the order of the summation.  Although the coefficients in the sum are unchanged, the new upshifted representative reflects the reversal of degrees in $z$.
\end{proof}

Of course, when $z=1$, the proposition just  restates the invariance of $v(M)$ under $G$.  On the other hand, if $z=-1$, we obtain an un-normalized analogue of the 3-$j$ symbol;  under the group action, $F(M, -1)$ is invariant up to sign. For the un-normalized Clebsch-Gordan coefficients, we have the following algorithm for computing Regge symmetries.  

\begin{corollary}  Let $g$ be an element of $G$.   To compute the effect of $g$ on $C(M)$, 
\begin{enumerate}
\item denote $g$ with respect to the permutation notation for $M$, 
\item select the corresponding equation from Proposition 7.3 and evaluate at $z=-1,$ and
\item rewrite using either the tensor product or 3-$j$ symbol notation.
\end{enumerate}  
\end{corollary}

\begin{example}  Consider the transformation that switches rows 1 and 3 in $M$.  The corresponding permutation is $(14)(25)(36)$, which gives $g\cdot {\bf a} =  (a_4, a_5, a_6, a_1, a_2, a_3)$ and
$${\bf a'} \ = \ (a_4+m_0,\ a_5+m_0,\ a_6+m_0,\ a_1-m_0,\ a_2-m_0, \ a_3-m_0).$$
Now
$$F(M, -1)\ =\ (-1)^{m_0} F(M', -1),$$
or
$$\ \begin{pmatrix} \rho(M) \\ a_1+a_5,\ a_2+a_6, \ a_3+a_4 \end{pmatrix}\ C(M)\ =\ (-1)^{a_2+a_5}\ \begin{pmatrix} \rho(M) \\ a_4+a_2,\ a_5+a_3, \ a_6+a_1 \end{pmatrix}\ C(M').$$ 
Converting to tensor product notation confirms formula (8.4) from \cite{Do1}:
\begin{equation*} \begin{pmatrix} m+n-k \\ k',\ n-j,\ m-i \end{pmatrix} c_{m, n, k}(i, j)\ = \  (-1)^i\  \begin{pmatrix}m+n-k \\ k,\ m-k,\ n-k \end{pmatrix}\ c_{m', n', k'}(i, j)\end{equation*}
where $m'=n-k+i,\ n'=m-k+j,\ $ and $k'=i+j-k$. 
\end{example}

\section{Endnotes}
Early results on  semi-magic squares appear in the landmark work of MacMahon \cite{PMM}, who gave the first formula for enumerating such squares of size three with fixed line sum. In addition to the Birkhoff and von Neumann references, see also \cite{BRo},  \cite{BZ},   \cite{Bo},  \cite{Lo1}, \cite{Lo2}, \cite{StCC} and \cite{StE}.  For the theory of hypergeometric functions and series, see, for instance, \cite{AAR} or \cite{Vi}. The survey \cite{JCP}, an evaluation of the work of Jacques Raynal,  includes both historical notes on Clebsch-Gordan coefficients and 3-$j$ symbols in quantum theory and an extensive bibliography. Our use of wreath products is elementary, essentially adapting the brief treatment in \cite{StC}.   For combinatorial reciprocity, for instance, as a feature of Ehrhart theory, see \cite{BRo}, \cite{BSa}, and \cite{StE}.

After posting a first draft of this work, we discovered section 1.4 of \cite{Lo1}, which arrives at a similar goal of Clebsch-Gordan coefficients as a theory of semi-magic squares, and typically our section 7 recasts some of the technique found there. Surprisingly, the two programs are mostly disjoint and beg reconciliation.    Together \cite{Lo1} and \cite{Lo2} put forth a coherent framework of combinatorial themes related to symmetry in quantum physics beyond just Clebsch-Gordan theory.

\bibliographystyle{amsplain}

\end{document}